\documentclass[11pt]{article}
\usepackage{graphicx}
\usepackage{amsthm}
\usepackage{amsfonts}
\usepackage{amsmath}
\usepackage{amssymb}
\usepackage{csquotes}
\usepackage{hyperref}
\usepackage{cleveref}
\usepackage{here}
\usepackage{fullpage}
\usepackage{cases}
\usepackage{algorithm}
\usepackage{algorithmicx}
\usepackage{algpseudocode}
\usepackage{caption}

\MakeOuterQuote{"}

\newtheorem{thm}{Theorem}[section]
\newtheorem{lem}[thm]{Lemma}

\newtheorem{dfn}[thm]{Definition}

\crefname{thm}{Theorem}{Theorems}
\crefname{lem}{Lemma}{Lemmas}
\crefname{pro}{Proposition}{Propositions}
\crefname{dfn}{Definition}{Definitions}
\crefname{exm}{Example}{Examples}
\crefname{rem}{Remark}{Remarks}
\crefname{figure}{Figure}{Figures}
\crefname{section}{Section}{Sections}

\title{5-Coloring Planar Graphs with a Color Class of Order at Most $|V|/6$}
\author{
    Yuta Inoue\thanks{The University of Tokyo, Tokyo, Japan, \texttt{yutainoue@is.s.u-tokyo.ac.jp}.}
    \and
    Ken-ichi Kawarabayashi\thanks{National Institute of Informatics \& The University of Tokyo, Tokyo, Japan, \texttt{k\_keniti@nii.ac.jp}.}
    \and
    Atsuyuki Miyashita\thanks{The University of Tokyo, Tokyo, Japan, \texttt{miyashita-atsuyuki869@is.s.u-tokyo.ac.jp}.}
}
\date{\today}

\begin{document}

\maketitle

\begin{abstract}
    We show that any planar graph $G=(V,E)$ has a 5-coloring such that one color class contains at most $|V|/6$ vertices.
    In other words, there exists a partition of $V$ into five independent sets $\{V_1, \cdots, V_5\}$ such that $|V_5| \leq |V| / 6$.
    Our proof yields an $O(|V|^2)$-time algorithm to find such a partition, and unlike the Four Color Theorem, our proof is fully verifiable without computer assistance.
\end{abstract}

\section{Introduction}

The Four Color Theorem (4CT), first shown by Appel and Haken \cite{4ct1,4ct2}, states that any loopless planar graph is 4-colorable. 
A simplified proof was later provided by Robertson et al.~\cite{RSST}. 
Due to the vast number of cases involved, the verification of the proof relies on computer programs, and to date, no proof independent of computer assistance is known.
This heavy reliance on computational verification is characteristic of much of the research that stems from the 4CT, such as \cite{edwards2016doublecross}, \cite{inoue2024projective}, and \cite{inoue2025toroidal}.
In contrast, the Five Color Theorem (5CT) \cite{heawood1890} admits a very short, human-verifiable proof \cite{MT}.

This paper aims to bridge the gap between the 4CT and 5CT by solving an intermediate problem: one that is more challenging than the 5CT but does not require computer assistance for its proof. 
One existing line of research in this direction focuses on finding large independent sets in planar graphs. 
The 5CT trivially guarantees an independent set of size at least $|V|/5$. 
This bound was improved by Albertson \cite{Albertson1976} to $2|V|/9$, and more recently to $3|V|/13$ by Cranston and Rabern \cite{Cranston2016}, both without the aid of computers. 
Other related work has focused on strengthening the 5CT itself. 
For instance, Thomassen \cite{MT,THOMASSEN1994180} provided an elegant proof of the 5-list-coloring theorem, and Chiba, Nishizeki, and Saito \cite{CHIBA1981317} developed a linear-time algorithm for 5-coloring planar graphs.

We propose a different approach to generalizing the 5CT. 
Recall that a \emph{5-coloring} of a simple planar graph $G=(V,E)$ is a partition of its vertices into five independent sets, or \emph{color classes}, $\{V_1, \dots, V_5\}$. 
Our goal is to prove the existence of a 5-coloring where the size of the smallest color class (say $V_5$) is bounded.
While the pigeonhole principle and 5CT imply a trivial bound of $|V_5| \leq \frac{1}{5}|V|$, 4CT implies that one color class can be empty ($|V_5|=0$). 
To bridge this gap, we seek to establish a non-trivial bound, $|V_5| \leq k|V|$ for some constant $k < 1/5$, using a proof that is entirely human-checkable. 
We aim for a proof that, while more technical than that of the 5CT, is significantly simpler than the proof of the 4CT.

We answer this affirmatively for $k = \frac{1}{6}$, giving a relatively simple and entirely human-checkable proof of the following theorem.

\begin{thm}\label{main}
     Any planar graph $G=(V,E)$ has a 5-coloring such that one color class contains at most $|V|/6$ vertices.
    In other words, there exists a partition of $V$ into five independent sets $\{V_1, \cdots, V_5\}$ such that $|V_5| \leq |V| / 6$.
\end{thm}

Our proof utilizes the method of reducible configurations and discharging, both of which are central to the proof of the 4CT. 
However, we employ a significantly smaller set of rules and configurations.
Specifically, our argument relies on 10 families of reducible configurations and 5 discharging rules (compared to 633 reducible configurations and 32 rules for the proof of 4CT \cite{RSST}).
On top of this, the reducibility criterion that we employ, which we call \emph{1/6-reducibility}, is significantly easier to check by hand compared to the existing reducibility criteria used in 4CT, such as D-reducibility and C-reducibility.

This paper is organized as follows. 
\Cref{sec:preliminaries} presents basic definitions. 
\Cref{sec:reducibility} defines reducibility and introduces the set of reducible configurations used in our proof. 
\Cref{sec:discharging} details the discharging argument, which completes the proof of Theorem \ref{main}. 
In \cref{sec:algorithm}, we present an algorithm that finds the desired 5-coloring in $O(|V|^2)$ time. 
Finally, \cref{sec:conclusion} offers concluding remarks.

\section{Preliminaries}\label{sec:preliminaries}
Let $G = (V,E)$ be a planar graph that we want to color.
In this paper, we assume that $G$ is loopless (or else it is impossible to obtain a proper coloring) and simple (the removal of one of the duplicated edges will not change its colorability).
Furthermore, we assume $G$ is maximal: no edges can be added to $G$ without breaking its planarity.
This can be done because after adding edges, a coloring of the resulting graph still yields a proper coloring of the original graph.
The embeddings of maximal planar graphs only have triangular faces; hence it is called a \emph{triangulation}.

Since $G$ is a simple triangulation, it is 3-connected (if $|V(G)| \ge 4$) and can be uniquely embedded on the sphere (plane).
We denote the degree of $v$ in $G$ by $d_G(v)$.
We denote the \emph{local rotation} of $v$ in $G$ by $\pi_v$, which is a cyclic permutation of all the edges incident to $v$, ordered counterclockwise on the embedding of $G$.
We denote the collection of all local rotations of $v \in V(G)$ by $\Pi$, which we call the \emph{rotation system}.

We now introduce \emph{configurations}.
A configuration is essentially a subgraph of a triangulation $G$, but it specifies not just the graph structure but also how it is embedded in $G$.

\begin{dfn} \label{dfn:configuration}
    Let $H = (V_H,E_H)$ be a connected subgraph of $G$ (not necessarily induced).
    We call $\mathcal{H} = (H, d_G|_{V_H}, \Pi_H)$ a \emph{configuration} in $G$, where:
    \begin{itemize}
        \item $H$ is called the \emph{underlying graph} of $\mathcal{H}$.
        \item $d_G|_{V_H}$ is the restriction of $d_G: V(G) \to \mathbb{Z}_{\ge0}$ to $V_H$.
        In later sections, we may just write it as $d_G$.
        \item $\Pi_H$ is a collection of local rotations of $\pi_v$ over all $v \in V_H$, but for all edges not in $E_H$, each such edge is replaced by a newly assigned element which we call a \emph{halfedge}. The resulting cyclic permutation is denoted as $\pi_v^H$.
    \end{itemize}
\end{dfn}

An example of a configuration is given in \cref{fig:config-example}.
For the values of $d_G$ for each node in $V_H$, we use the notation introduced by Heesch \cite{heesch1969}, which we show in \cref{fig:heesch}.

\begin{figure}[H]
  \centering
  \includegraphics[width=250pt]{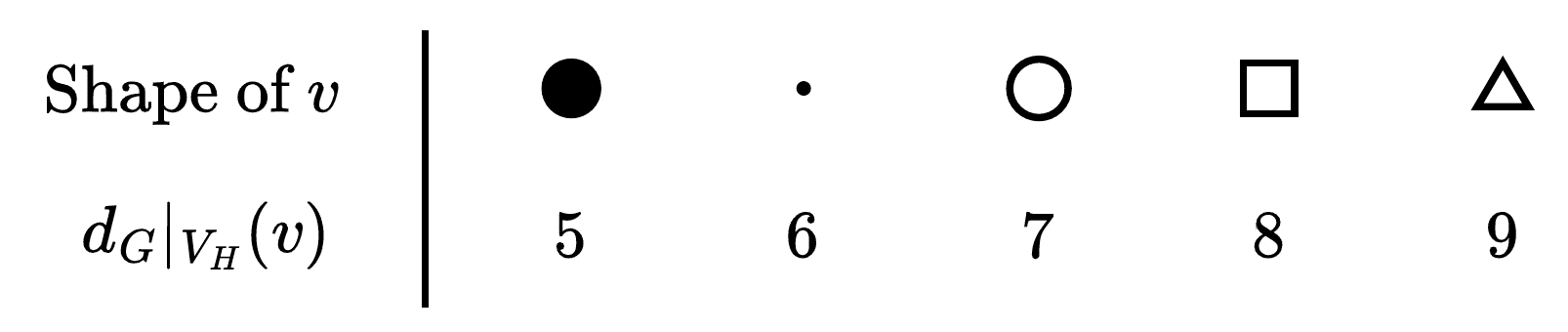}
  \caption{The chart for the Heesch notation. We depict vertices of degree 5 by a filled circle, 6 by a dot, 7 by an unfilled circle, 8 by a square, and 9 by a triangle.}
  \label{fig:heesch}
\end{figure}

\begin{figure}[H]
  \centering
  \includegraphics[width=100pt]{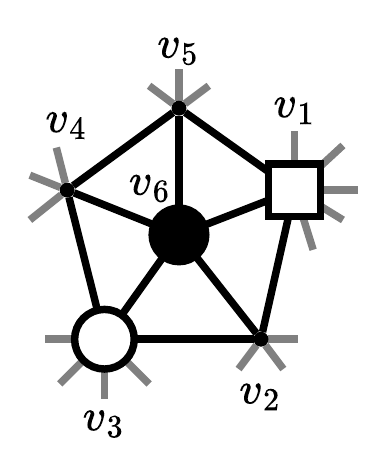}
  \caption{A drawing of an example configuration $\mathcal{H} = (H, d_G|_{V_H}, \Pi|_{V_H})$.
    The vertices are depicted by Heesch notation (corresponding to the value of $d_G$, which is $d_G(v_1) = 8, d_G(v_3) = 7, d_G(v_6) = 5, d_G(v_2) = d_G(v_4) = d_G(v_5) = 6$).
    The gray lines represent halfedges in $\Pi_H$.
    For example, $\pi_{v_1}^H$ is the cyclic permutation $(v_1v_5, v_1v_6, v_1v_2, f_{11}, f_{12}, f_{13}, f_{14}, f_{15})$, $\pi_{v_2}^H$ is the cyclic permutation $(v_2v_1, v_2v_6, v_2v_3, f_{21}, f_{22}, f_{23})$, and so on. (Here, $f_{ij}$ is some halfedge connected to $v_i$).}
  \label{fig:config-example}
\end{figure}

For some pair of configurations $\mathcal{H}_1 = (H_1, d_G|_{V_1}, \Pi|_{H_1})$ and $\mathcal{H}_2 = (H_2, d_G|_{V_2}, \Pi_{H_2})$, even if $H_1$ and $H_2$ are isomorphic and have matching degrees with respect to $d_G|_{V_1}$ and $d_G|_{V_2}$, if the local rotation $\Pi_{H_1}$ and $\Pi_{H_2}$ differ, we must distinguish these two configurations.
This is especially important when $H_1, H_2$ both have a cutvertex, as shown in \cref{fig:config-pi-difference}.

\begin{figure}[H]
  \centering
  \includegraphics[width=150pt]{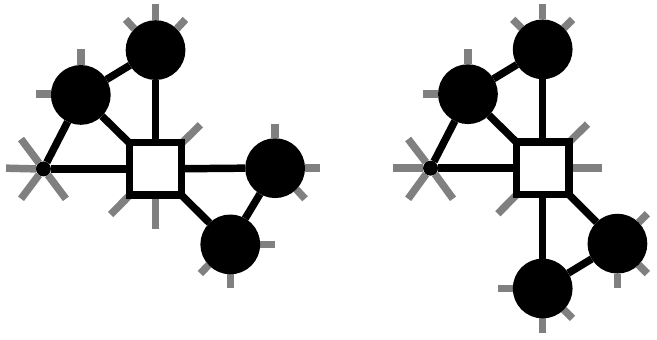}
  \caption{An example of two configurations having isomorphic graphs and matching degrees but different local rotation on the central vertex.}
  \label{fig:config-pi-difference}
\end{figure}

On the other hand, if $H$ does not contain a cutvertex, there is a trivial local rotation for all of the vertices (excluding symmetry).
In these cases, we may omit $\Pi_H$ and just say $\mathcal{H} = (H, d_G)$ is a configuration. 

\section{Reducibility}\label{sec:reducibility}

Our goal is to find a configuration $\mathcal{H} = (H, d_G, \Pi_H)$ in $G$ so that the vertices in $H$ can always be colored given that the vertices in $G - H$ are pre-colored. (For now, let us consider the problem with only four colors.)
This property is referred to as \emph{reducibility}\footnote{In 4CT (and similar works), specific types of reducibility such as \emph{D-reducibility} and \emph{C-reducibility} are checked for various configurations. Due to its complexity, these reducibility criteria are very tedious and cumbersome to prove without the aid of computer programs, so we do not use them at all in this paper.} of graphs in previous works.
The reducibility of a single vertex of small degree is straightforward. If a vertex $v$ has $d_G(v) \leq 3$, it can be colored with a color not used by its neighbors.

When $G$ has a single vertex $v$ of $d_G(v) = 4$, we try to color $v$ with an available color, but this may not be possible when the four neighbors all are assigned to exactly one of the four colors (that is, colors that are not the fifth color).
In this case, we utilize \emph{Kempe chains}.
$i,j$-Kempe chains are connected components of a colored graph induced by the union of two color classes $V_i$ and $V_j$.
By reassigning all vertices in a particular $i,j$-Kempe chain that are colored in $i$ to $j$ and vice versa (commonly known as the \emph{Kempe change}), we can obtain a different proper coloring of $G$.

The key observation here is that two Kempe chains, each being a $i,j$-Kempe chain and a $k,l$-Kempe chain respectively (assuming $i,j,k,l$ are all different colors) do not overlap, so when a four cycle $v_1v_2v_3v_4$ exists in a colored planar graph $G$ such that they all are assigned different colors, it is not possible that the diagonally crossed pairs $v_1,v_3$ and $v_2,v_4$ are both connected via Kempe chains (which would violate planarity of $G$). 
Using this technique, one can alter the color of only a certain neighboring vertex around $v$.
After the Kempe changes, the neighbors of $v$ only have three distinct colors assigned, and hence $v$ can be colored with an unused color from the four colors.

This method of reducing vertices of degree four or less can be applied multiple times.
For example, a configuration $\mathcal{H}$ of two vertices $v_1, v_2$ with $d_G(v_1) = 4, d_G(v_2) = 5$ can be regarded as reducible, since after deleting $v_1$ from $G$ (where $\mathcal{H}$ appears), $v_2$ will have degree four in the remaining graph $G - v_1$.
We generalize this method and define the following reducibility criteria called \emph{$0$-reducibility}. 
(The naming "0"-reducibility comes from the fact that the vertices inside $H$ can be colored without using the fifth color, in contrast to the $1/6$-reducibility that we will define later.)

\begin{dfn}\label{dfn:0-reducibility}
    Let $(H, d_G, \Pi_H)$ be a configuration in $G$.
    We say it is \emph{$0$-reducible} if the vertices in $H$ can be sequentially deleted so that all vertices are of degree at most four at the time of deletion.
\end{dfn}

Now, let us introduce the fifth color.

In this paper, we aim to 5-color the entire graph $G = (V,E)$ with at most $|V| / 6$ vertices colored with the fifth color.
Hence, a reducibility of a configuration suffices with 5-coloring the vertices in the underlying graph $H$ with at most $\lfloor |H| / 6 \rfloor$ vertices colored with the fifth color.
When $|V(H)| \geq 6$, the impact is massive: we are allowed to assign some vertex in $H$ with the fifth color, relaxing the coloring constraint significantly.
Therefore, we define a new, more lax reducibility criterion for these graphs, which we name \emph{1/6-reducibility}.

The key idea here is that we introduce a sequence of vertices that we color the fifth color, each being the backup plan in case all vertices preceding fail to be colored by the fifth color.
We believe following an example is best suited for laying out the ideas, so we do this first and state the formal definitions afterwards.

Let us consider the configuration drawn in \cref{fig:config-example}, where $H$ is a graph of order $6$, and $d_G(v_1),\cdots,d_G(v_6)$ equals $8,6,7,6,6,5$, respectively. 
(We encourage the reader to refer to \cref{fig:1/6-example} as well while reading this passage.)

Let us assume there already is a 5-coloring of the vertices in $ G-V (H)$.
Then, we first consider $v_1$, which has degree $d_G(v_1) = 8$.
$v_1$ has three neighbors inside $H$, so there must be five neighbors outside $H$.
Among those five neighbors, if there are no vertices colored with the fifth color, we can assign the fifth color to $v_1$.
Then, the rest of the graph $H_1$ (with the degree function $d_{G-v_1}$ mapping $v_2, \cdots, v_6$ to $5,7,6,5,4$, respectively) is $0$-reducible and can be 4-colored.
\footnote{Here we use Kempe changes, which potentially alter the sizes of the color classes dramatically. 
However, we never use Kempe chains that involve the fifth color, so the bound on the size of this color class remains intact.}

Conversely, if $v_1$ has an outer neighbor colored with the fifth color, we look at another vertex, $v_3$ of degree $d_G(v_3) = 7$.
$v_3$ also has three neighbors inside $H$, so there must be four neighbors outside $H$.
Among those four neighbors, if there are no vertices colored with the fifth color, we can assign the fifth color to $v_3$.
In this case, the rest of the graph $H_3$ (with $d_{G - v_3}$ mapping $v_1, v_2, v_4, v_5, v_6$ to $8,5,5,6,4$, respectively) is unfortunately not $0$-reducible, because of $v_1$ having too large a degree.
However, in this case, we still have a way of proving reducibility. 
Deleting vertices of degree four in $H$ recursively, we are left with one vertex $v_1$ having degree $5$.
Remember that $v_1$ had a neighbor (outside $H$) having the fifth color. (Say, $u_1$.)
We want to color $v_1$ with one of the four colors, and we have at most four neighbors of $v_1$ that are assigned a non-fifth color. 
Therefore, $v_1$ can essentially be regarded as a degree $4$ vertex, and the Kempe change strategy we discussed with degree 4 vertices works here as well!
Using the reducibility terms, one can say that the graph $H_3$ with $d_{G - v_3 - u_1}$ ($u_1$ being the vertex adjacent to $v_1$ colored with the fifth color) assigning $v_1, v_2,v_4,v_5,v_6$ to $7,5,5,6,4$, respectively, is $0$-reducible. 

We proceed to the case where both $v_1$ and $v_3$ cannot be colored with the fifth color due to their outer neighbors.
We now start with $v_2$ and color this with the fifth color, if possible.
The remaining graph (with the degrees in $G - v_2 - u_1 - u_3$, where $u_3$ is a vertex adjacent to $v_3$ colored with the fifth color) is 0-reducible.
(Here, $u_3$ may be identical to $u_1$, but the same argument holds.)

The final case is when $v_1,v_3,v_2$ all have neighboring vertices colored with the fifth color.
In this case, we choose $v_6$ as the vertex to color with the fifth color.
The remaining graph (with the degrees in $G - v_6 - u_1 - u_3 - u_2$, where $u_2$ is a vertex adjacent to $v_2$ colored with the fifth color) is 0-reducible, and fortunately, $v_6$ does not have any neighbors outside $H$, meaning that we can always color $v_6$ regardless of the colorings outside $H$.
(Again, $u_2$ may be identical to $u_1$ or $u_3$, but the same argument holds.)
Hence, we finally have the scheme to color the vertices in $H$ with at most one vertex having the fifth color, regardless of how the vertices in $G-H$ are colored.

We depict the discussion above as graph images in \cref{fig:1/6-example}.

\begin{figure}[H]
  \centering
  \includegraphics[width=300pt]{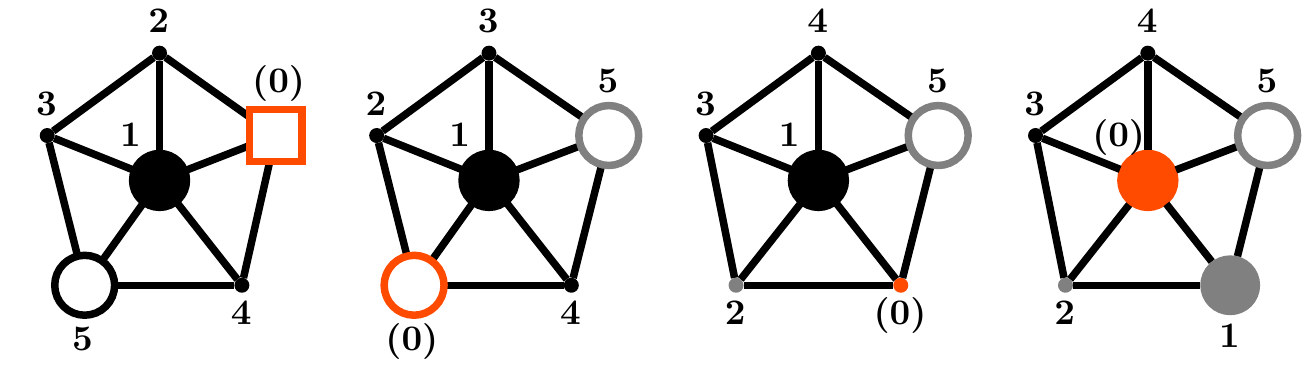}
  \caption{A visualization of how the example configuration should be 5-colored, shown from left to right.
  The red vertex (labeled \textbf{(0)}) of each graph is the vertex that we color with the fifth color, and the gray vertices are the vertices that we know have a neighboring vertex of the fifth color (hence, having $1$ less degree compared to the previous graphs).
  We can confirm that all the non-red-vertices form a 0-reducible configuration, by deleting the 5 vertices in labeled order from $\mathbf{1}$ to $\mathbf{5}$ and by seeing that they are degree 4 or less at the time of deletion.}
  \label{fig:1/6-example}
\end{figure}

1/6-reducibility can be defined formally as follows.

\begin{dfn}\label{dfn:1/6-reducibility}
    Let $\mathcal{H} = (H, d_G, \Pi_H)$ be a configuration with $|V(H)| \ge 6$. 
    We say $\mathcal{H}$ is \emph{$1/6$-reducible} if there exists a finite sequence called the \emph{trial sequence} of its vertices $(v_1, v_2, \cdots, v_k)$ that guarantees a valid 5-coloring of $H$ with at most one vertex using the fifth color, given any 5-coloring of $G - V(H)$.
    
    The coloring strategy is as follows: For a given coloring of $G - V(H)$, let $U_5$ be the set of vertices in $G - V(H)$ with the fifth color.
    \begin{enumerate}
        \item Find the first vertex $v_i$ in the trial sequence that $v_i$ is not adjacent to any vertex in $U_5$.
        \item If such a vertex $v_i$ is found, color it with the fifth color.
        The proof of reducibility requires showing that the remaining $H - v_i$ is 0-reducible (and thus 4-colorable), considering that for every $j < i$, the vertex $v_j$ has its degree effectively lowered by its adjacency to at least one vertex in $U_5$.
        \item If all vertices $v_1, \cdots, v_k$ are adjacent to a vertex in $U_5$, $H$ as a whole must be $0$-reducible (and 4-colorable) after the degrees are lowered by the removal of vertices in $U_5$.
    \end{enumerate}
\end{dfn}

To conclude, we introduce one lemma that directly follows from the definition of $0$ and $1/6$-reducibility. 

\begin{lem}
    Let there exist two isomorphic graphs $H$ and $H'$ such that $d_G(v) \leq d_G(v')$ for every corresponding vertex pair $v \in V_H, v' \in V_{H'}$.
    Then, if the configuration $(H', d_G|_{V_{H'}}, \Pi_{H'})$ is $0$ or $1/6$-reducible, the configuration $(H, d_{G}|_{V_H}, \Pi_{H})$ is also $0$ or $1/6$-reducible, respectively.
    \label{lem:dominance}
\end{lem}

\subsection{List of \texorpdfstring{$0$, $1/6$}{0, 1/6}-reducible configurations}
We now prove the reducibility of several configurations.

The first lemma is for $0$-reducibility and is trivial.

\begin{lem}\label{lem:four-degree}
    Let $H = (\{v\}, \varnothing)$ be a singleton vertex of degree at most four.
    Then, $(H, d_G)$ is $0$-reducible.
\end{lem}

For the upcoming $1/6$-reducible configurations, we provide a figure for each and every configuration.
The images are equipped with the information of the trial sequence, labeled with parenthetical numbers. ((1), (2), etc.)

\begin{lem}\label{lem:five-wheel}
    Let $H$ be a graph containing 6 vertices $v_1, \cdots, v_6$ such that $v_1, \cdots, v_5$ induces a 5-cycle and $v_6$ is adjacent to all of $v_1, \cdots, v_5$.
    Let $d_G(v_6) = 5$, and the multiset $\{d_G(v_1), \cdots, d_G(v_5)\}$ equals to $\{6,6,6,7,8\}$.
    Then, $(H, d_G)$ is $1/6$-reducible.
\end{lem}

\begin{proof}
    By rotation and inversion symmetry, we only need to check for the following two cases: when the degree 8 vertex and the degree 7 vertex are adjacent, and when they are separated by a degree 6 vertex (both cases drawn in \cref{fig:five-wheel}).
    The case where they are separated is exactly the example we have shown in the Preliminaries in \cref{fig:1/6-example,fig:config-example}.
    The case where the two vertices are adjacent can be proved similarly, by using the trial sequence $v_1, v_2,v_3,v_6$, where $d_G(v_1) = 8, d_G(v_2) = 7, d_G(v_6) = 5$, and $v_3$ is any vertex of degree 6.
\end{proof}

\begin{figure}[H]
  \centering
  \includegraphics[width=60pt]{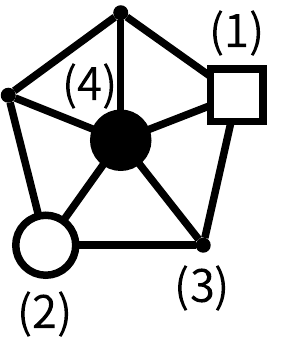}
  \hspace{25pt}
  \includegraphics[width=60pt]{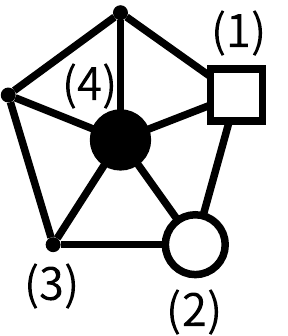}
  \caption{The two cases to consider in \cref{lem:five-wheel}.}
  \label{fig:five-wheel}
\end{figure}

\begin{lem} \label{lem:seven-with-four-fives-one-eight}
    Let $H$ be a graph that consists of 6 vertices $v_1, \cdots, v_6$ such that $d_G(v_1) = 7$, $d_G(v_2) = 8$, $d_G(v_3), \cdots, d_G(v_6) = 5$, and $v_2,\cdots, v_6$ are all neighbors of $v_1$.
    Additionally, let $v_3, v_4$ be neighbors of $v_2$.
    Then, $H$ is $1/6$-reducible.
\end{lem}

\begin{proof}
    By \cref{dfn:configuration}, the overall structure of the configuration is one of the graphs drawn in \cref{fig:seven-with-four-fives-one-eight}.
    In either one of these cases, $1/6$-reducibility can be shown by the trial sequence $v_2, v_1, v_i$ (where $v_i$ is one of the vertices of degree 5 adjacent to another vertex of degree 5 in $H$).
\end{proof}

\begin{figure}[H]
  \centering

  \includegraphics[width=240pt]{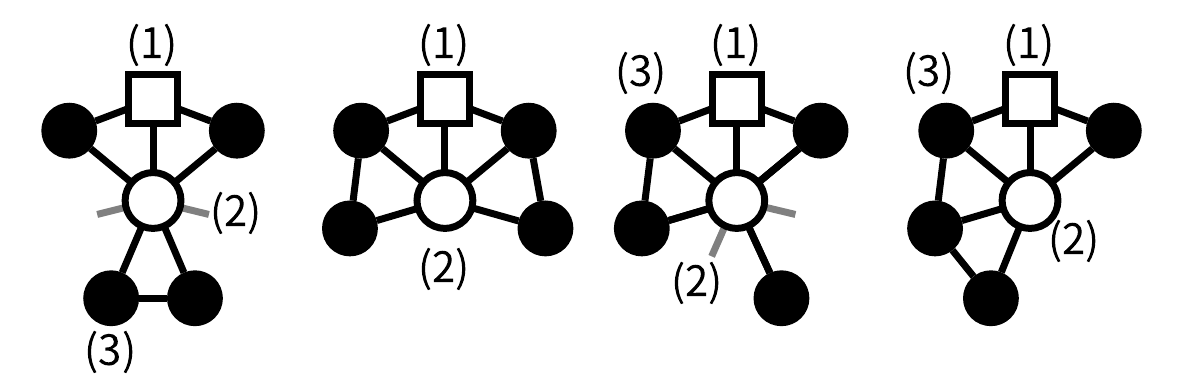}
  \caption{The four cases to consider in \cref{lem:seven-with-four-fives-one-eight}.}
  \label{fig:seven-with-four-fives-one-eight}
\end{figure}

\begin{lem} \label{lem:five-with-four-fives-plus-alpha}
    Let $H$ be a graph that consists of 6 vertices $v_1, \cdots, v_6$, such that $v_1$ is adjacent to all of $v_2, \cdots, v_5$, $v_2v_3v_4v_5$ is a path of order four, and $v_6$ is connected in such a way that $\mathcal{H} = (H, d_G, \Pi)$ is one of the three configurations drawn in \cref{fig:five-with-four-fives-plus-alpha}.
    Then, $\mathcal{H}$ is 1/6-reducible.
\end{lem}

\begin{proof}
    We can show each of its reducibility by the trial sequence $v_6, v_1, v_2$.
\end{proof}

\begin{figure}[H]
  \centering
  \includegraphics[width=200pt]{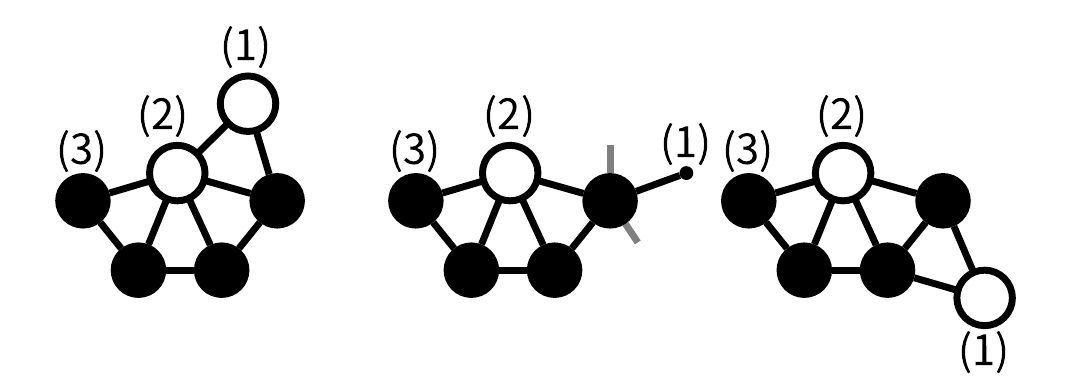}
  \caption{The three cases to consider in \cref{lem:five-with-four-fives-plus-alpha}. (Note that $d_G(v_6)=7$ in the left and right cases, whereas $d_G(v_6) = 6$ in the center case.)}
  \label{fig:five-with-four-fives-plus-alpha}
\end{figure}

\begin{lem} \label{lem:2/3-with-a-five}
    The four configurations depicted in \cref{fig:2/3-with-a-five} are $1/6$-reducible.
\end{lem}

\begin{proof}
    Every configuration has a trial sequence of length 3, starting from the degree 7 vertex, followed by a degree 6 vertex adjacent to the previous degree 7 vertex (if there are multiple, whichever is fine), and finally the degree 5 vertex adjacent to the previous two vertices.
\end{proof}

\begin{figure}[H]
  \centering
  \includegraphics[width=260pt]{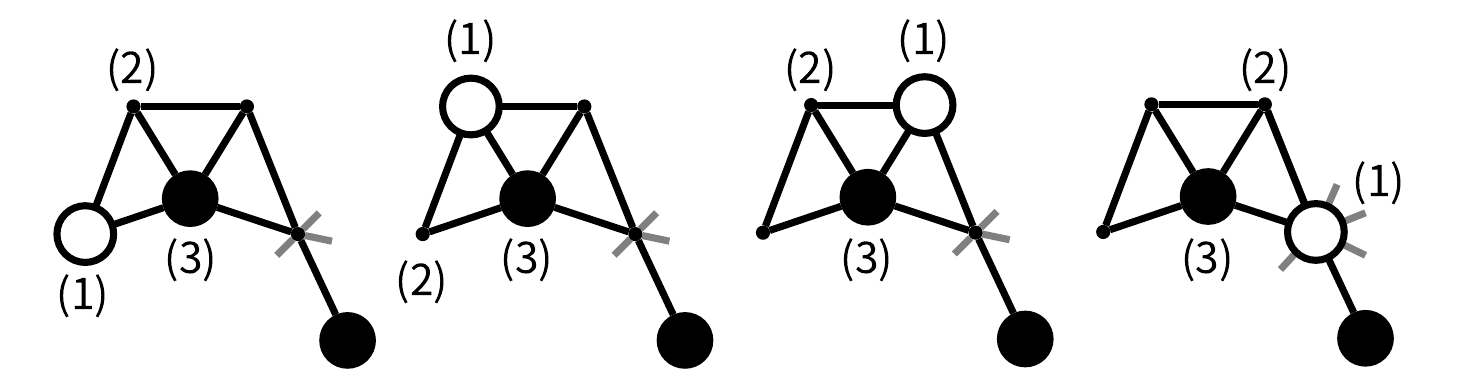}
  \caption{The four cases to consider in \cref{lem:2/3-with-a-five}.}
  \label{fig:2/3-with-a-five}
\end{figure}

\begin{lem} \label{lem:1-with-a-1/2}
    The two configurations depicted in \cref{fig:1-with-a-1/2} are $1/6$-reducible.    
\end{lem}

\begin{proof}
    Both of the configurations have trial sequences of length two, consisting of a degree 6 vertex followed by a degree 5 vertex.
\end{proof}

\begin{figure}[H]
  \centering
  \includegraphics[width=160pt]{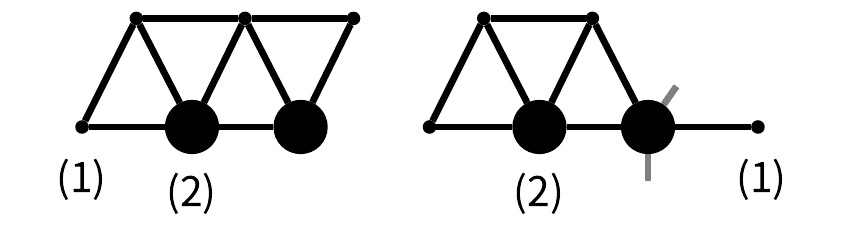}
  \caption{The two cases to consider in \cref{lem:1-with-a-1/2}.}
  \label{fig:1-with-a-1/2}
\end{figure}

\subsection{Virtually $1/6$-reducible configurations}

We finally introduce a class of configurations that are technically not $1/6$-reducible, but have a coloring scheme that is almost identical to the strategy we use in $1/6$-reducible configurations.
We say that they are \emph{virtually $1/6$-reducible} and treat them as reducible configurations. 
We explain their definition along with the proof of the following lemma.

\begin{lem} \label{lem:large-d-with-d-3-fives}
    Let $d \geq 8$ be an integer and let $H$ be a graph that consists of $d-2$ vertices $v$ and $v_1,\cdots,v_{d-3}$ such that $v$ is adjacent to all of $v_1,\cdots,v_{d-3}$.
    Assume $d_G(v) = d$ and $d_G(v_i) = 5$ for all $i = 1,\cdots, d-3$.
    Then, there is a way to color the vertices in $H$ with at most one vertex assigned the fifth color, given a precoloring of $G - H$, where $\mathcal{H} = (H, d_G, \Pi_H)$ appears in $G$.
\end{lem}

\begin{proof}
    Let us name the neighbors of $v$ as $u_0,u_1,\cdots,u_{d-1}$ (ordered counterclockwise), indexed with $i \in \mathbb{Z}/d\mathbb{Z}$.
    Since the edge $u_iu_{i+1}$ exists for every $i$ in $G$, there are at most three connected components in $H-v$.
    The connected components can be expressed as a range over the cycle $u_0u_1\cdots u_{d-1}$, separated by $u_i, u_j, u_k \notin V(H)$.
    
    We now provide the coloring scheme.
    Assume that $G - V(H)$ is already colored.
    We first try to color $v$ with the fifth color, if none of $u_i, u_j, u_k$ are colored with the fifth color.
    After this, $d_{G-v}(v_i) = 4$ for all $v_i \in V(H)$, so $(H-v, d_{G-v})$ is 0-reducible.
    
    If one of $u_i, u_j, u_k$ is colored with the fifth-color, we use a backup plan. 
    Fix $u_i$ as the vertex with the fifth color without loss of generality.
    We consider the following three cases (See \cref{fig:large-d-with-d-3-fives} for a visualization.).

    \begin{itemize}
        \item When $H-v$ is connected: Let us select some $v' \in V(H-v)$. 
        The configuration $(H-{v'}, d_{G - v' - u_i})$ is $0$-reducible, since the removal of $v'$ will allow all vertices of degree 5 to be deleted sequentially. Note that if we cannot color $v'$ with the fifth-color, then $v'$ also becomes degree four, and hence we can delete all the vertices of degree five. 
        \item When $H-v$ has two connected components: At least one of them must be adjacent to $u_i$, which we call $H_1$. 
        Let the other connected component be called $H'$.
        Then, by selecting some $v' \in V(H')$, the configuration $(H-v',d_{G-v'-u_i})$ is $0$-reducible, since the vertices in $H_1$ will be deleted due to $u_i$ being removed, and the removal of $v'$ will allow all vertices in $H'$ to be deleted sequentially. Again, if we cannot select such a vertex to give the fifth-color, we can also delete all the vertices in $H'$. 
        \item When $H-v$ has three connected components: At least two of them must be adjacent to $u_i$, which we call $H_1, H_2$.
        Let the third connected component be called $H'$.
        Then, by selecting some $v' \in V(H')$, the configuration $(H-v',d_{G-v'-u_i})$ is $0$-reducible, since the vertices in $H_1$ and $H_2$ will be deleted due to $u_i$ being removed, and the removal of $v'$ will allow all vertices in $H'$ to be deleted sequentially.  Again, if we cannot select such a vertex to give the fifth-color, we can also delete all the vertices in $H'$. 
    \end{itemize}

    Therefore, the vertices of $H$ can be 5-colored with only one vertex assigned to the fifth color.
\end{proof}

\begin{figure}[H]
  \centering
  \includegraphics[width=300pt]{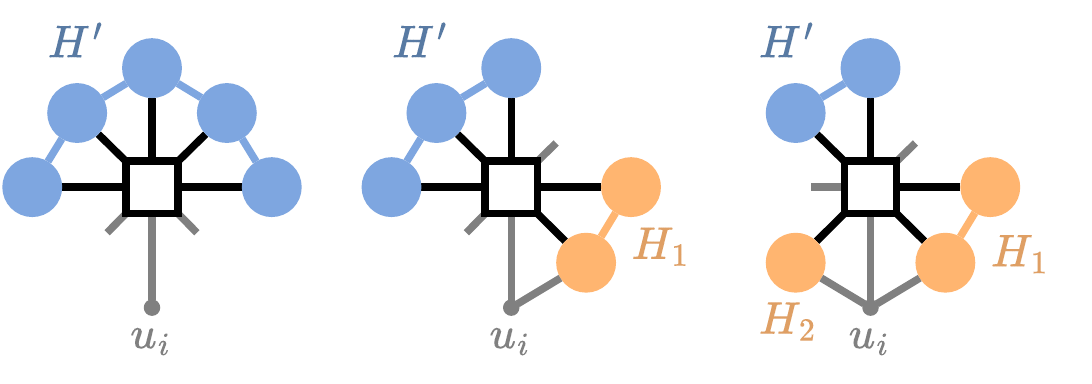}
  \caption{The examples for the three cases for when $d = 8$.}
  \label{fig:large-d-with-d-3-fives}
\end{figure}

\section{Discharging}\label{sec:discharging}

\subsection{Discharging rule}
We now apply the \emph{discharging method} to prove that one reducible configuration appears in any given planar (triangulation) graph $G$.

We initially assign a value called the \emph{charge} to every vertex $v \in G$, equal to $6 - d_G(v)$. 
Then, for every vertex $v$ and its neighbor $u$, we send $c(v, u)$ charge from $v$ to $u$.
The value of $c(v, u)$ is determined by the \emph{discharging rule} defined as follows and depicted in \cref{fig:rules}.
(We denote $N_d(v)$ as the set of neighbors of $v$ of degree exactly $d$.)

\begin{itemize}
    \item Rule A: When $d_G(v) = 5$ and $d_G(u) \ge 7$,
    \begin{numcases}{c(v, u) := }
        1/3 & $(d_G(u) = 7)$ \tag{A-1}
        \\
        1/2 & $(d_G(u) = 8)$ \tag{A-2}
        \\
        \max(1/3, r(v)/|N_{9+}(v)|) & $(d_G(u) \ge 9)$ \tag{A-3}
    \end{numcases}
    Here, $r(v)$ is the remaining charge of $v$ after distributing charges to neighbors of degree $8$ or less, which equates to $1 - |N_7(v)|\times1/3 - |N_8(v)|\times1/2$ ($N_{9+}(v)$ is the set of neighbors of $v$ of degree 9 or above).
    \item Rule B: When $d_G(v) = 7$, and $|N_5(v)| = 4$.
    \begin{itemize}
        \item If the neighbors of $v$ in $N_5(v)$ do not form a path: $c(v,u) = 1/3$ if $u \in N_{9+}(v)$ and $u$ is adjacent to two vertices in $N_5(v)$.
        \item If the neighbors of $v$ in $N_5(v)$ do form a path: $c(v,u) = 1/6$ if $u$ is the endpoint of the path.
    \end{itemize}
    \item Otherwise, $c(v,u) = 0$.
\end{itemize}


\begin{figure}[H]
    \centering
    \includegraphics[width=150pt]{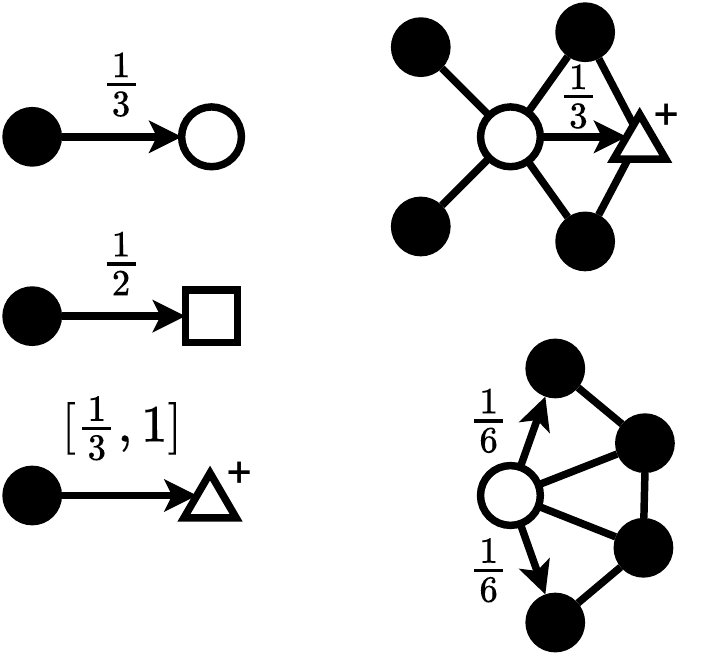}
    \caption{The subgraphs where positive charge is being sent. The three cases on the left correspond to Rule A, and the two cases on the right correspond to Rule B.}
    \label{fig:rules}
\end{figure}

The final charge of each vertex $v$ is equal to the following:
\begin{equation*}
    c(v) = (6-d_G(v)) - \sum_{u \in N(v)} (c(v, u) - c(u, v))
\end{equation*}

The summation of all the charges of the vertices $\sum_v c(v) = 6 |V| - \sum_v d_G(v) = 6 |V| - 2 |E|$, which equals $12$ by Euler's formula. 
(Note that the charge movements $- \sum_{u \in N(v)} (c(v, u) - c(u, v))$ cancel out when they are summed up over all the vertices.)
Therefore, the following lemma holds.

\begin{lem}\label{lem:exists-positive-charge}
    For any triangulation $G$, there is at least one vertex with $c(v) > 0$.
\end{lem}

\subsection{Separating Cycle}
In the following section, in order to prove the existence of a particular configuration $\mathcal{H} = (H, d_G, \Pi)$ in $G$, we typically find a set of $h = |V(H)|$ vertices $\{v_1,\cdots,v_h\}$, and check that the degree matches $d_G$, the local rotation matches $\Pi$, and the edges in between $v_1,\cdots,v_h$ comply with the structure of $H$.

We also need to check that the vertices $v_1,\cdots,v_h$ are distinct, that is, no two vertices $v_i$ and $v_j$ are actually the same.
This is trivial when $H$ is of diameter 2, since an identified vertex with a path of order one or two between them implies that $G$ contains a loop or a multiple edge.
On the other hand, when $v_i$ and $v_j$ that are of distance $3$ in $H$ exist, we cannot say for sure that $v_i$ and $v_j$ are different, since separating 3-cycles are allowed in $G$.

Still, we can show that some degree constraints hold if such a 3-cycle exists. 

\begin{lem}\label{lem:three-cycle-18-plus}
    Let $u, v, w$ be three vertices in $G$ such that $uv, vw, wu \in E(G)$, but $uvw$ is not a facial cycle.
    Then, $d_G(u) + d_G(v) + d_G(w) \geq 18$, or a vertex of degree 3 exists in $G$.
\end{lem}

\begin{proof}
    Let us fix a connected component $G'$ in $G - \{u,v,w\}$.
    Let us denote the number of edges connecting $u$ and a vertex in $G'$ as $n_u$, and let us define $ n_v$ and $ n_w$ likewise.

    We now use the fact that $G$ is a triangulation and $G$ is simple.
    Suppose $n_u + n_v + n_w < 6$.
    The only cases where $G$ is simple are the following: $|V(G')|=1$ where $n_u+n_v+n_w = 3$, or $|V(G')|=2$ where $n_u + n_v + n_w = 5$.
    In either case, $G'$ contains a vertex of degree $3$ in $G$.

    We can prove this for both sides of the connected component of $G - \{u,v,w\}$.
    This sets a lower bound of the number of edges coming out of the three vertices, and the sum of the degrees of $u$, $v$, and $w$ is at least $6 + 6 + 6 = 18$.
\end{proof}

Looking at all reducible configurations $\mathcal{H} = (H, d_G, \Pi)$ of diameter 3 or more which we use in our paper (\cref{lem:seven-with-four-fives-one-eight,lem:five-with-four-fives-plus-alpha,lem:2/3-with-a-five,lem:1-with-a-1/2}), we can see that for all vertex pairs $v_i, v_j$ of distance exactly $3$ in $H$, there exists a path $v_iu_1u_2v_j$ such that $\min\{d_G(v_i),d_G(v_j)\} + d_G(u_1) + d_G(u_2) < 18$, and the edges $v_iu_1$ and $u_2u_1$ are not consecutive in $\pi(u_1)$, and the edges $u_1u_2$ and $v_ju_2$ are not consecutive in $\pi(u_2)$.
This, along with \cref{lem:three-cycle-18-plus}, derives the following.

\begin{lem}\label{lem:4.2}
    Let $G$ be a triangulation of minimum degree $4$ or more.
Suppose that a reducible configuration $\mathcal{H} = (H, d_G, \Pi_H)$ defined in one of \cref{lem:five-wheel,lem:seven-with-four-fives-one-eight,lem:five-with-four-fives-plus-alpha,lem:2/3-with-a-five,lem:1-with-a-1/2,lem:large-d-with-d-3-fives} exists in $G$ such that there is a graph homomorphism (not necessarily a graph isomorphism) $\phi$ from $H$ to $(V=\{v_1,\cdots,v_k\} \subset V(G), E \subset E(G))$ (satisfying the degree condition by $d_G$).
    Then, the homomorphism $\phi$ is actually an isomorphism, i.e. $v_1, \cdots, v_k$ must be pairwise distinct, and $G[V]$ induces a graph that contains $H$ as a subgraph.
\end{lem}

\subsection{Unavoidability}

Hereafter, we assume that $G$ is a \emph{minimal counterexample} of \cref{main}, i.e. a triangulation with the smallest order that contradicts \cref{main}.
Let a reducible configuration $\mathcal{H}$ be contained in $G$.
Then, the graph $G - V(H)$ (where $H$ is the underlying graph of $\mathcal{H}$) must be colorable with respect to \cref{main} by the minimality of $G$.
By reducibility of $\mathcal{H}$, the vertices in $H$ can also be colored; hence, a contradiction.

Therefore, there must be no reducible configuration in $G$.

Now, we will show the resulting charge $c(v) \le 0$ for all $v \in V(G)$. (This contradicts \cref{lem:exists-positive-charge}, thus proving \cref{main} by contradiction.)
In the following proofs, let $v_0, v_1, \cdots, v_{d-1}$ ($d := d_G(v)$) be the neighbors of $v$, indexed with $\mathbb{Z}/d\mathbb{Z}$, such that $v_0v_1\cdots v_{d-1}$ forms a cycle ordered counterclockwise in $G$.

Hereafter, we divide the proof into five cases, $d_G(v) = 5,6,7,8,$ and $9+$. 
In each case, we will show that $v$ ends up with a non-positive charge when there is no reducible configuration in $G$.
Note that, by \cref{lem:4.2}, it suffices to find a vertex set having the same structure and degree count to ensure that it exists in $G$.

\subsubsection{$d_G(v) = 5$}
The vertex $v$ initially has charge $1$.
First, by Rule A, $v$ sends some charge to the neighboring vertices of degree $7$ or more.
Combining \cref{lem:dominance,lem:five-wheel}, at least one of the following is satisfied.
\begin{itemize}
    \item One neighbor $v_i$ has $d_G(v_i) \ge 9$.
    $v$ sends charge $1$ to $v_i$.
    \item Two neighbors $v_i, v_j$ has $d_G(v_i) \ge 8, d_G(v_j) \ge 8$.
    $v$ sends charge $1/2$ or more to each of $v_i, v_j$.
    \item Three neighbors $v_i, v_j, v_k$ has $d_G(v_i) \ge 7, d_G(v_j) \ge 7, d_G(v_k) \ge 7$.
    $v$ sends charge $1/3$ or more to each of $v_i, v_j, v_k$.
\end{itemize}
In either of these cases, $v$ sends a total charge $1$ or more to the neighboring vertices, resulting in a non-positive charge.

Next, by Rule B, $v$ may receive $1/6$ charge from some neighbors of degree $7$.
We now prove the following.

\begin{lem} \label{lem:ring-of-seven-and-four-fives}
    Let $v$ receive $1/6$ charge by Rule B from a vertex $v_i$ of degree 7.
    Then, $v$ has four neighbors (including $v_i$) of degree at least 7, in which at least two of them are of degree at least 8.
\end{lem}

\begin{proof}
    By Rule B, the four vertices of degree 5 neighboring $v_i$ are aligned in a path. 
    Let this path be $u_1u_2u_3u_4$. 
    Additionally, $v$ is the endpoint of this path (i.e., $v=u_1$).
    Since $v_iu_1u_2 = v_ivu_2$ is a facial triangle, $u_2$ is identical to $v_{i-1}$. 
    (The positioning of each vertex is depicted in \cref{fig:charge-of-seven-with-four-fives}.)
    The remaining three neighbors of $v$ are $v_{i+1}$, $v_{i+2}$, and $v_{i+3}$.
    
    Now, if $d_G(v_{i+1}) \leq 7$, the vertices $v_i, u_1, u_2, u_3, u_4$ along with $v_{i+1}$ constitute a $1/6$-reducible configuration by \cref{lem:five-with-four-fives-plus-alpha,lem:dominance}.
    Similarly, if $d_G(v_{i+3}) \leq 7$ or $d_G(v_{i+2}) \leq 6$, the vertices $v_i, u_1, u_2, u_3, u_4$ along with $v_{i+3}$ or $v_{i+2}$ constitute a $1/6$-reducible configuration, respectively. 

    Therefore, the three inequalities $d_G(v_{i+1}) \geq 8$, $d_G(v_{i+3}) \geq 8$, and $d_G(v_{i+2}) \geq 7$ hold.
\end{proof}

With \cref{lem:ring-of-seven-and-four-fives}, we know the lower bound of the degree of the surrounding vertices of a vertex $u_1$ sending $1/6$ charge back to $v$, as shown in \cref{fig:charge-of-seven-with-four-fives}. (By symmetry, both the endpoints of the paths of order four that consist of degree 5 are given $1/6$ charge.)
From this argument, one can confirm that $v$ sends out a net minimum of charge one, even after the $1/6$ charge received from $v_i$, by giving out a total of $1/3 \times 4 = 4/3$ charge when Rule A is applied. Thus, $v$ ends up with a non-positive charge. 

\begin{figure}[H]
  \centering
  \includegraphics[width=350pt]{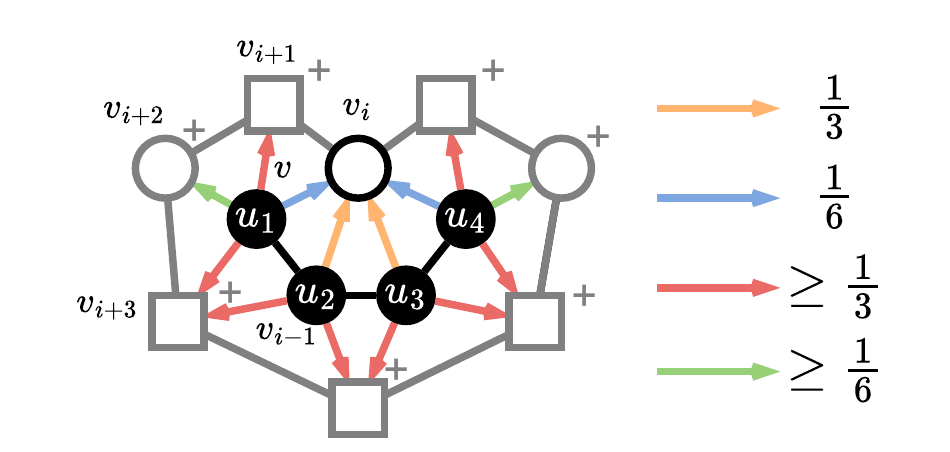}
  \caption{A vertex of degree 7 surrounded by four vertices of degree 5, aligned in a path. Here, the overall charge movement (value of $c(u,v)-c(v,u)$) of all edges from the degree 5 vertices is shown by color, where red sends at least $1/3$, orange sends $1/3$, green sends at least $1/3$, and blue sends $1/6$. 
  The vertex of degree 5 sends a minimum of $1/6 + 1/6 + 1/3 + 1/3 = 1$ charge, whereas the vertex of degree 7 receives $1/6 + 1/3 + 1/3 + 1/6 = 1$ charge.}
  \label{fig:charge-of-seven-with-four-fives}
\end{figure}

\subsubsection{$d_G(v) = 6$}
The vertex $v$ has an initial charge $0$, and no charge is sent to $v$. 
Thus $v$ ends up with charge 0. 

\subsubsection{$d_G(v) = 7$}
The vertex $v$ has an initial charge $-1$. Recall that $N_5(v) := \{v_i \mid d_G(v_i) = 5\}$.

When $|N_5(v)| \le 3$, the total charge that will be sent from them is at most $1/3 \times 3 = 1$, resulting in a non-positive charge.
On the other hand, by \cref{lem:seven-with-four-fives-one-eight,lem:dominance}, when $|N_5(v)| \ge 5$, $v$ along with the neighbors, will exhibit $1/6$-reducibility.

The remaining case is $|N_5(v)| = 4$, where the vertices in $N_5(v)$ send total charge $1/3 \times 4 = 4/3$.
Suppose a vertex $u = v_i$ such that $d_G(v_{i-1}) = d_G(v_{i+1}) = 5$ exists.
When $d_G(u) \le 8$, by \cref{lem:seven-with-four-fives-one-eight,lem:dominance}, the four vertices in $N_5(v)$ along with $v$ and $u$ consist of a $1/6$-reducible configuration.
When $d_G(u) \ge 9$, this vertex receives $1/3$ charge, and the resulting charge of $v$ is exactly $0$.

On the other hand, when no such $u$ exists, the induced graph $G[N_5(v)]$ must be a path, i.e. an $i$ with $N_5(v) = \{v_i,v_{i+1},v_{i+2},v_{i+3}\}$ exists.
In this case, $1/6$ charge is sent to each of $v_i$ and $v_{i+3}$, and the resulting charge of $v$ is exactly $0$. (See the vertex of degree 7 in \cref{fig:charge-of-seven-with-four-fives}.)





\subsubsection{$d_G(v) = 8$}
The vertex $v$ has an initial charge of $-2$.
When there are four or fewer neighbors of degree 5, the total charge that will be sent from them is $1/2 \times 4 = 2$ or less, resulting in a non-positive charge.
On the other hand, by \cref{lem:large-d-with-d-3-fives} (with $d = 8$), when there are five or more neighbors of degree 5, $v$ along with the neighbors will exhibit $1/6$-reducibility.

\subsubsection{$d_G(v) \geq 9$}
First, let us prove some lemmas regarding the charges being sent to $v$. Looking back at the discharging rule, the following can be directly shown.

\begin{lem}\label{lem:sending-to-nine-plus}
    $v$ only receives charge from neighboring vertices of degree $5$ or $7$. The amount of charge from a vertex of degree $7$ will always be $1/3$ or $0$, and the amount of charge from a vertex of degree $5$ (let us call it $u$) can be determined by the following case analysis:
    
    Let us remind that $N_{7+}(u)$ is the set of neighbors of $u$ of degree at least $7$. Then,
    \begin{itemize}
        \item If $N_{7+}(u) = \{v\}$, $v$ receives $1$ charge from $u$.
        \item If $N_{7+}(u) = \{v, v'\}$ and $d_G(v') = 7$, $v$ receives $2/3$ charge from $u$.
        \item If $N_{7+}(u) = \{v, v'\}$ and $d_G(v') \ge 8$, $v$ receives $1/2$ charge from $u$.
        \item Otherwise, $v$ receives no more than $1/3$ charge from $u$.
    \end{itemize}
\end{lem}

We also prove the following.

\begin{lem}\label{lem:charge-seven-to-nine-plus}
    When $v$ receives $1/3$ charge from a vertex $v_i$ of degree 7, $v_{i-1}$ and $v_{i+1}$ must both be of degree $5$, and they each only send at most $1/3$ charge to $v$.
\end{lem}

\begin{proof}
    The fact that $v_{i-1}$ and $v_{i+1}$ are of degree $5$ directly follows from the statement of Rule B.
    If $v_{i-1}$ sends more than $1/3$ charge, this means $v_{i-1}$ has only two neighbors of degree at least 7, by \cref{lem:sending-to-nine-plus}.
    The other three neighbors of degree at most $6$ are aligned in such a way that it matches the rightmost configuration depicted in \cref{fig:2/3-with-a-five}, in which the reducibility is proved in \cref{lem:2/3-with-a-five}.
    Therefore, for no reducible configuration to appear, $v_{i-1}$ can only send charge at most $1/3$.
    By symmetry, the same is true for $v_{i+1}$.
\end{proof}

Let the number of vertices of degree $5$ neighboring $v$ be $m$.
If $m \leq d_G(v) - 6$, even if all of the vertices of degree $5$ send charge $1$, which is the maximum charge possible by \cref{lem:sending-to-nine-plus}, the resulting charge of $v$ will be $0$, a non-positive charge.
(Some vertices of degree $7$ may send charge too, but it hinders neighboring vertices from sending large charge by \cref{lem:charge-seven-to-nine-plus}, thus resulting in the overall charge sent to $v$ even smaller.)
On the other hand, if $m \geq d_G(v) - 3$, the vertices of degree $5$ along with $v$ produce a 1/6-reducible configuration.

Therefore, the only cases to consider are when $m = d_G(v) - 5$ and $m = d_G(v) - 4$.

Now, let us denote the amount of charge sent from $v_i$ to $v$ as $c_i := c(v_i, v) - c(v, v_i)$, indexed with $i \in \mathbb{Z} / d_G(v)\mathbb{Z}$.

\begin{lem}\label{lem:charge-sum-leq-one}
    If $v_i$ and $v_{i+2}$ are both of degree $5$, $c_i + c_{i+2} \leq 1$.
\end{lem}

\begin{proof}
    By \cref{lem:dominance,lem:2/3-with-a-five}, if $2/3 \le c_i \le 1$ and $d_G(v_{i+2}) = 5$, we have a reducible configuration.
    Therefore,  $c_i \leq 1/2$, and by symmetry, $c_{i+2} \leq 1/2$ too.
    Therefore, the lemma holds.
\end{proof}

We can use \cref{lem:charge-sum-leq-one} (along with \cref{lem:charge-seven-to-nine-plus}) to bound the amount of charge that $v$ is receiving, $\sum_i c_i$.
Notably, for $m = d_G(v) - 5$, it suffices to find one pair of vertices $v_i$ and $v_{i+2}$ that are both of degree 5, and for $m = d_G(v) - 4$, it suffices to find two disjoint pairs of vertices $v_i$, $v_{i+2}$, and $v_{j}$, $v_{j+2}$ that are all of degree 5.

From here, we provide a proof specialized for $d_G(v) = 9$, and a general proof for all $d_G(v) \geq 10$.

\begin{enumerate}
\renewcommand{\labelenumi}{(\alph{enumi})}
    \item \textbf{When $d_G(v) = 9, m = d_G(v) - 5 = 4$:} We can assume no $i$ exists such that $v_i$ and $v_{i+2}$ are both of degree 5.
    This means that no two neighboring vertices in the cyclic permutation $\pi_9 = (v_0,v_2,v_4,v_6,v_8,v_1,v_3,v_5,v_7)$ are both of degree 5.
    Excluding rotational symmetry, the only solution is when $v_0, v_4, v_8, v_3$ are the four vertices.
    Even then, two of the four vertices must send charge $1$ for the sum to exceed $3$.
    However, by \cref{lem:1-with-a-1/2}, $c_0 + c_8 \leq 4/3$ and $c_3 + c_4 \leq 4/3$. 
    As a result, the total charge does not exceed $3$.
    \item \textbf{When $d_G(v) = 9, m = d_G(v) - 4 = 5$:} We can assume no $i, j$ exists such that the four distinct vertices $v_{i}$, $v_{i+2}$, $v_{j}$, and $v_{j+2}$ are all of degree 5.
    Looking again at the cyclic permutation $\pi_9 = (v_0,v_2,v_4,v_6,v_8,v_1,v_3,v_5,v_7)$, the only way to select 5 vertices with at most one neighboring pair in $\pi_9$ (excluding symmetry) is to select $\{v_0, v_2, v_6, v_1, v_5\}$.
    These five vertices, along with $v$, constitute a configuration that is similar to the configurations defined in \cref{lem:large-d-with-d-3-fives}, but having $d-3$ vertices instead of $d-2$ vertices.
    Still, we can say that this is (virtually) $1/6$-reducible because the five vertices induce two connected components, $\{v_0,v_1,v_2\}$ and $\{v_5,v_6\}$, and this is the second case of the case analysis shown in \cref{fig:large-d-with-d-3-fives}.
    \item \textbf{When $d_G(v) \geq 10, m = d_G(v) - 5$:}
    We use a similar method to case (a), by constructing a cyclic permutation and then selecting $m$ elements from it so that no two elements are adjacent in the cyclic permutation.
    When $d_G(v)$ is odd, we have the cyclic permutation \begin{equation*}\pi_{d_G(v)} = (v_0, v_2,\cdots,v_{d_G(v)-1},v_1,v_3,\cdots, v_{d_G(v)-2})\end{equation*} and then selecting $m = d_G(v) - 5$ elements without neighboring pairs in $\pi_{d_G(v)}$ is impossible given $d_G(v) \geq 11$.
    When $d_G(v)$ is even, we need two cyclic permutations:\begin{align*}\pi_{d_G(v)}^{\textrm{even}} &= (v_0, v_2, \cdots, v_{d_G(v)-2})\\ \pi_{d_G(v)}^{\textrm{odd}} &= (v_1, v_3, \cdots, v_{d_G(v) - 1})\end{align*}
    This is equivalent to trying to select $d_G(v) - 5$ vertices as a stable set from two cycles of length $d_G(v) / 2$, which is also impossible given that $d_G(v) / 2 \geq 5$.
    \item \textbf{When $d_G(v) \geq 10, m = d_G(v) - 4$:}
    We combine the techniques used in cases (b) and (c).
    That is, we select $m$ vertices from $v_{d_G(v)}$ or $v_{d_G(v)}^{\textrm{odd}}, v_{d_G(v)}^{\textrm{even}}$, depending on the parity of $d_G(v)$, such that no two disjoint pairs of neighboring vertices exist.
    By simple proof using stable sets on cycles, we can also prove that such a selection is impossible.
\end{enumerate}

In summary, $v$ always ends up with a non-positive charge, and this completes the proof of \cref{main}.\qed

\section{Algorithm}\label{sec:algorithm}

We are given a planar graph $G_{\text{input}}$.
If the number of vertices is 3 or fewer, we return a trivial 3-coloring of the graph.
If not, the first step is to fill in extra edges so that they become a triangulation.
This can be done by finding an embedding of $G_{\text{input}}$, iterating over all faces of size $4$ or larger, and inserting edges to triangulate the faces.
The resulting graph $G$ should be simple.

Next, we try to find a reducible configuration in $G$.
We first check whether a vertex $v$ of degree at most $4$ exists.
If it does, the single vertex $v$ is $0$-reducible, and after obtaining the coloring of $G-v$, we can color $v$ with one of the four colors. (See \cref{dfn:0-reducibility}.)

If all vertices have degree at least $5$, we now go on to search reducible configurations that we listed in \cref{lem:four-degree,lem:five-wheel,lem:seven-with-four-fives-one-eight,lem:five-with-four-fives-plus-alpha,lem:2/3-with-a-five,lem:1-with-a-1/2,lem:large-d-with-d-3-fives}.
Since the local rotation of each vertex is fixed in a configuration, one can iterate over all vertices and their incident edges in $G$, map them to a certain vertex and edge in a configuration, and check whether the rest of the structure of the configuration matches with $G$.
If the size of each configuration is constant-bounded, the whole process can be done in linear time.
The exception to this is the infinite class of configurations described in \cref{lem:large-d-with-d-3-fives}, the vertex count being $d-2$ for each $d \ge 8$.
For this, we have a workaround, by iterating over all vertices $v$ of degree at least $8$, checking the degrees of the neighboring $d$ vertices, and finding whether all but at most $3$ vertices have degree $5$.
Since the average degree $2|E|/|V|$ is constant bounded (less than $6$), this iteration can be done in linear time as well.

We are guaranteed to find a 1/6-reducible configuration $\mathcal{H}$.
Hereafter, we proceed to color the inner vertices by the method described in \cref{dfn:1/6-reducibility}.

The running time of the algorithm $T(n)$ where $n$ is the input size (proportional to $|V|$) can be described as $T(n) = O(n) + T(n - k)$ for $k \ge 1$: we recursively call the coloring algorithm on a strictly smaller graph (size $n-k$, where the reducible configuration has $k$ vertices), and we need $O(n)$ time to find a configuration and color its inner vertices.
Therefore, the total running time $T(n)$ equates to $O(n^2)$.

\section{Conclusion}\label{sec:conclusion}

This work opens several avenues for further research.
The most immediate question is whether the bound of $k = 1/6$ can be improved further, perhaps to $k = 1/7$ or less.
If we attempt to use "$1/m$-reducible" configurations (analogous to $1/6$-reducible configurations), the graph size (at least $m$) will increase as $k = 1/m$ becomes smaller. 
Hence, it seems inevitable that both the complexity of the discharging rules and the reducible configuration sets will increase by a significant margin.
In particular, since the configuration in \cref{lem:five-wheel} is no longer reducible for larger $m$, we likely would need additional rules to send charges to vertices of degree $6$.

Another interesting problem will be to tackle a similar problem on the dual graph.
The dual formulation of 4CT is that every 2-connected cubic planar graph is 3-edge-colorable \cite{tait1880dual}, while the 4-edge-colorability is trivial by Vizing \cite{vizing1965critical}.
Therefore, the intermediate problem will be to find a 4-edge-coloring over a 2-connected cubic planar graph such that one color class has at most $k'|E|$ edges, where $k' < 1/4$.

\section*{Acknowledgements}

The first author was supported by JSPS Kakenhi 22H05001, JST ASPIRE JPMJAP2302 and the Hirose Foundation Scholarship.
The second author was supported by JSPS Kakenhi 22H05001, JSPS Kakenhi JP20A402 and JST ASPIRE JPMJAP2302.
The third author was supported by JSPS Kakenhi 22H05001, JST ASPIRE JPMJAP2302, and JST SPRING JPMJSP2108.

\bibliographystyle{abbrv}
\bibliography{citations}

\end{document}